\documentclass{gtmon_a}
\pdfoutput=1
\usepackage{pinlabel}
\usepackage[all]{xy}  %%% at the top

\proceedingstitle{The Zieschang Gedenkschrift}
\conferencestart{5 September 2007}
\conferenceend{8 September 2007}
\conferencename{Conference in honour of Heiner Zieschang}
\conferencelocation{Toulouse, France}

\editor{Michel Boileau}
\givenname{Michel}
\surname{Boileau}

\editor{Martin Scharlemann}
\givenname{Martin}
\surname{Scharlemann}

\editor{Richard Weidmann}
\givenname{Richard}
\surname{Weidmann}

\title{Roots in 3--manifold topology}

\author{C Hog-Angeloni}
\givenname{Cynthia}
\surname{Hog-Angeloni}

\email{cyn@math.uni-frankfurt.de}
\urladdr{http://www.math.uni-frankfurt.de/~angeloni}

\author{S Matveev}
\givenname{Sergei}
\surname{Matveev}
\address{Fachbereich Mathematik der Universit\"at Frankfurt\\\newline
Postfach 111932\\60054 Frankfurt\\Germany\vspace{3pt}\\\newline
Chelyabinsk State University\\129 Kashirin Brothers St.\\\newline
454021 Chelyabinsk\\Russia}
\email{matveev@csu.ru}
\urladdr{}

\dedicatory{}

\volumenumber{14}
\issuenumber{}
\publicationyear{2008}
\papernumber{13}
\startpage{295}
\endpage{319}

\doi{}
\MR{}
\Zbl{}

\arxivreference{}


\received{8 October 2006}
\revised{}
\accepted{19 March 2007}
\published{29 April 2008}
\publishedonline{29 April 2008}
\proposed{}
\seconded{}
\corresponding{}
\version{}

\let\xysavmatrix\xymatrix
\def\xymatrix{\disablesubscriptcorrection\xysavmatrix}
\AtBeginDocument{\let\bar\wbar\let\tilde\wtilde\let\hat\what}
\makeop{Con}

\newtheorem{lemma}{Lemma}
\newtheorem{theorem}{Theorem}

\newtheorem{corollary}{Corollary}
\theoremstyle{remark}
\newtheorem{definition}{Definition}
\newtheorem{remark}{Remark}
\newtheorem*{Examplenn}{Example}
\newtheorem{Example}{Example}
\newtheorem{Conjecture}{Conjecture}

\def\p{\partial}
\def\N{\mathbb N}
\def\Cl{{\mbox{Cl}}}
\def\Int{\mbox{Int}}

\begin{document}

\begin{asciiabstract}
Let C be some class of objects equipped with a set of simplifying
moves. When we apply these to a given object M in C as long as
possible, we get a root of M.

Our  main result is that under certain conditions   the root  of
any object exists and is unique. We apply this result to different
situations and get several new results and new proofs of known
results. Among them there are a new proof of the Kneser-Milnor
prime decomposition theorem for 3-manifolds and different versions
of this theorem for cobordisms, knotted graphs, and orbifolds.
\end{asciiabstract}

\begin{htmlabstract}
<p class="noindent"> Let C be some class of objects equipped with a
set of <em> simplifying moves</em>. When we apply these to a given
object M&isin;C as long as possible, we get a <em>root</em> of M.
</p> <p class="noindent"> Our main result is that under certain
conditions the root of any object exists and is unique. We apply this
result to different situations and get several new results and new
proofs of known results. Among them there are a new proof of the
Kneser&ndash;Milnor prime decomposition theorem for 3&ndash;manifolds
and different versions of this theorem for cobordisms, knotted graphs,
and orbifolds.  </p>
\end{htmlabstract}

\begin{abstract}
Let $\cal C$ be some class of objects equipped with a set of {\em
simplifying moves}. When we apply these to a given object $M\in \cal
C$ as long as possible, we get a {\em root} of $M$.

Our  main result is that under certain conditions   the root  of
any object exists and is unique. We apply this result to different
situations and get several new results and new proofs of known
results. Among them there are a new proof of the Kneser--Milnor
prime decomposition theorem for 3--manifolds and different versions
of this theorem for cobordisms, knotted graphs, and orbifolds.
\end{abstract}

\maketitle

\section{Introduction} Let $\cal C$ be some class of
objects equipped with a set  of {\em simplifying moves}. When we apply
these to a given object $M\in \cal C$ as long as possible, we get a {\em
root} of $M$.

Our  main result is that under certain conditions   the root  of
any object exists and is unique. We apply this result to different
situations and get several new results and new proofs of known
results. Among them there are a new proof of the Kneser--Milnor
prime decomposition theorem for 3--manifolds and different versions
of this theorem for cobordisms, knotted graphs, and orbifolds.

We thank C~Gordon, W~Metzler, C~Petronio, W\,H Rehn and  S~Zentner
for useful discussions as well as the referee for his helpful comments.

Both authors are partially supported by the INTAS Project
``CalcoMet-GT" 03-51-3662 and the second author is also partially
supported by the RFBR grant 05-01-00293.

\section{Definition, existence and uniqueness of a root}
Let $\Gamma$ be an oriented graph and $e$ an   edge of $\Gamma$ with
initial vertex $v$ and terminal vertex $w$. We will call the transition
from $v$ to $w$ an \emph{edge move} on $v$.

 \begin{definition} A vertex $R(v)$ of $\Gamma$ is \emph{a root}
 of $v$, if the
following holds:
\begin{enumerate}
\item  $R(v)$ can be obtained from $v$ by edge moves.

\item  $R(v)$ admits no further edge moves.
\end{enumerate}
\end{definition}

Recall that a set  $A$ is \emph{well ordered}  if any subset of $A$ has a
least element. Basic examples are the set of non-negative integers $\N_0$
and its power $\N_0^k$ with lexicographical order.

\begin{definition} Let $\Gamma$  be an oriented graph with
vertex set $V(\Gamma)$ and $A$ a well ordered set. Then a map $c\colon
V(\Gamma)\to A$ is called a \emph{complexity function}, if for any edge
$e$ of  $\Gamma$ with vertices $v$, $w$ and orientation from $v$ to $w$ we
have $c(v)>c(w)$.
\end{definition}

\begin{definition}  Let $\Gamma$ be an oriented graph.
 Then two edges $e$ and $d$ of $\Gamma$ with the same initial
vertex $v$ are called \emph{elementary equivalent}, if their endpoints
have a common root. They are called \emph{equivalent} (notation: $e\sim
d$), if there is a sequence of edges $e=e_1, e_2,\ldots,e_n=d$ such that
the edges $e_i$ and $e_{i+1}$ are elementary equivalent for all $i, 1\leq
i <n$.
\end{definition}

\begin{definition} \label{AB} Let $\Gamma$ be an oriented graph.
We say that $\Gamma$ possesses property (CF) if it admits a  complexity
function. $\Gamma$ possesses property (EE) if any two edges of $\Gamma$
with common initial vertex are equivalent.
\end{definition}

It turns out that property (CF) guarantees existence, while
property (EE) guarantees uniqueness of the root.

\begin{theorem} \label{main} Let $\Gamma$ be an oriented graph
possessing properties (CF) and (EE).  Then any vertex has a unique
root.
\end{theorem}

\begin{proof}{\bf Existence}\qua Let $v$ be a vertex of $\Gamma$.
Denote by $X$ the set of all vertices of $\Gamma$ which can be obtained
from $v$ by edge moves. By property (CF), there is a complexity function
$c\colon V(\Gamma)\to A$. Since $A$ is well ordered, the set $c(X)$ has a
least element $a_0$. Then any vertex in $c^{-1}(a_0)$ is a root of $v$.

{\bf Uniqueness}\qua  Assume that
$v$ is a least counterexample, ie, $v$ has two different roots
$u\neq w$ and   $c(v)\leq c(v')$ for any vertex $v'$ having more
than one root. Let $e$ respectively $d$  be the first edge of an
oriented edge path from $v$ toward $u$ respectively $w$. By
property  (EE), we have a sequence $e=e_1, e_2,\ldots,e_n=d$ such
that the edges $e_i$ and $e_{i+1}$ are elementary equivalent for
all $i, 1\leq i <n$. Hence, their endpoints $v_i, v_{i+1}$ have a
common root $r_i$. As $c(v_i)<c(v)$ for all $i$, that root is in
fact unique. Thus $u=r_1 =\cdots =r_n =w$ which is a
contradiction.\end{proof}

  The following sections are devoted to applications of \fullref{main}.

\section{A simple proof of the Kneser--Milnor prime decomposition theorem}
\label{milnor} \label{S-roots}
\begin{definition} \label{df:sphereduction} Let   $S$ be a
2--sphere in the interior of a compact 3--manifold $M$. Then a {\em
spherical reduction} (or \emph{compression move})  of $M$ along $S$
consists in cutting $M$ along $S$ and attaching two balls to the two
2--spheres arising under the cut.
\end{definition}

Let us consider the three types of spherical reductions. If $S$ bounds a
3--ball in $M$, then the reduction of $M$ along $S$ is {\em trivial}, ie,
it produces a copy of $M$ and a 3--sphere. If $S$ does not bound a 3--ball
and separates $M$ into two parts, then the reduction along $S$ produces
two 3--manifolds $M_1, M_2$ such that $M=M_1\# M_2$,  the connected sum of
$M_1, M_2$. If $S$ does not separate $M$, then the reduction along $S$
produces a 3--manifold $M_1$ such that $M=M_1\# S^2{\times}S^1$ or
$M=M_1\# S^2 {\tilde \times} S^1$ (the latter is possible only if $M$ is
non-orientable).

Recall that a 3--manifold  $M$ is prime if it is not a connected
sum of two 3--manifolds different from $S^3$. Also, $M$ is
irreducible, if it admits no non-trivial spherical reductions.

\begin{theorem}[Kneser--Milnor prime decomposition
 \cite{Kn,Mil}]\label{Milnor}
 Any closed orientable 3--manifold $M$ is a connected sum of prime factors.
The  factors   are determined uniquely up to homeomorphism.
\end{theorem}

The same holds for compact 3--manifolds with boundary. For non-orientable
3--manifolds the Kneser--Milnor theorem is also true, but with the following
modification: the factors are determined uniquely up to homeomorphism and
replacement of the direct product $ S^2{\times}S^1$ by the skew
product $S^2 {\tilde \times S^1} $ and vice versa. Note that these two
products are the only 3--manifolds that are prime, but reducible.

To get into the situation of \fullref{main}, we introduce an oriented
graph $\Gamma$. The set of vertices of $\Gamma$ is defined to be the
set of all compact 3--manifolds considered up to homeomorphism and
removing all connected components homeomorphic to $S^3$. Two manifolds
$M_1,M_2$ are joined by an oriented edge from $M_1$ to $M_2$ if and
only if the union of non-spherical components of $M_2$ can be obtained
from $M_1$ by a non-trivial spherical reduction and removing spherical
components.

Our next goal is to prove that any vertex of $\Gamma$ has a unique
root.

\begin{remark}
Our choice of vertices of $\Gamma$ is important. If we did not neglect
spherical components, then the root would not be unique. For example, the
manifold $S^2{\times}S^1\# S^2{\times}S^1$ would have two
different roots: a 3--sphere and the union of two disjoint 3--spheres.
\end{remark}

In order to construct a complexity function, we need Kneser's
Lemma\footnote{Here is Kneser's original statement of the lemma: Zu jeder
$M^3$ gehoert eine Zahl $k$ mit der folgenden Eigenschaft: Nimmt man mit
$M^3$ nacheinander $k +1$ Reduktionen vor, so ist mindestens eine davon
trivial. Durch $k$ (oder weniger) nicht triviale Reduktionen wird $M^3$ in
eine irreduzible $M^3$ verwandelt.}, see~\cite{Kn}.

\begin{lemma} \label{boundg}
For any compact 3--manifold $M$ there exists an integer
constant $C_0$ such that
  any sequence of non-trivial   spherical reductions
consists of no more than $C_0$ moves.\endproof
\end{lemma}

\begin{corollary} \label{comfun} $\Gamma$ possesses property (CF).
\end{corollary}
\begin{proof} We define the  complexity function $c\colon V(\Gamma) \to
\N_0$ as follows: $c(M)$ is the maximal number of spherical reductions in
any sequence of non-trivial spherical reductions of $M$. By \fullref{boundg},
$c$ is well defined and evidently it is compatible with the
orientation on $\Gamma$.
\end{proof}

\begin{lemma} \label{edgequiv} $\Gamma$ possesses property (EE).
\end{lemma}
\begin{proof} Let $S_e ,S_d$ be two non-trivial spheres in
$M$ and  $e, d$ be the corresponding edges of $\Gamma$. We  prove
 the equivalence $e\sim d$ by induction on the number
 $m=\# (S_e\cap S_d)$ of curves
 in the intersection assuming that the  spheres have been
 isotopically shifted so that $m$ is minimal.

 {\bf Base of induction}\qua  Let $m=0$, ie~$ S_e\cap S_d=\emptyset$.
 Denote by $M_e,M_{d}$ the manifolds obtained by reducing $M$
 along $S_e,S_d$, respectively. Since $ S_e\cap S_d=\emptyset$,
 $S_d$ survives the reduction along $S_e$ and thus may be considered as
 a sphere in $M_e$. Let $N$ be obtained by reducing
 $M_e$ along $S_d$. Of course,
 compression of $ M_{d} $  along  $S_e $
 also gives $ N $.   We claim that any root $R$ of $N$ is a
 common root of  $ M_{e} $  and $ M_{d}$ (and hence the edges
 $e,d$ are
 elementary equivalent).
Indeed,   if the sphere $S_d$ is non-trivial in $M_e$, then $R$ is a root
of $M_e$ by the definition of a root. If $S_d$ is trivial in $M_e$, then
$N=M_e\cup S^3$ and the manifolds $M_e,N$ determine the same vertex of
$\Gamma$, so again $R$ is a root of $M_e$. Symmetrically, $R$ is a root of
$M_d$ whether or not $S_e$ is trivial in $M_d$.\\
Therefore,  $R$ is a root of $M_e$ and $M_d$.

{\bf Inductive step}\qua Suppose that $m>0$. Using an innermost circle
argument, we find a disc $D\subset S_d$ such that $D\cap S_e=\p D$ and
compress $S_e$ along $D$.  By minimality of $m$, we get two non-trivial
spheres $S',S''$, each disjoint to $S_e$ and intersecting $S_d$ in a
smaller number of circles, see \fullref{maintrick}.
Taking one of them
(say, $S'$) and denoting the corresponding edge by $e'$, we get $e \sim
e'$ and $e'\sim d$ by the inductive assumption. Therefore, $e\sim
d$.\end{proof}

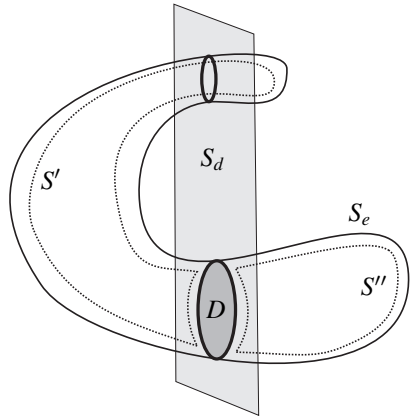
\begin{figure}[ht]
\labellist
\small
\pinlabel {$S'$} [l] at 150 673
\pinlabel {$S_d$} at 260 685
\pinlabel {$D$} at 262 593
\pinlabel {$S_e$} [b] at 350 640
\pinlabel {$S''$} [r] at 372 608
\endlabellist
\centerline{\includegraphics[height=5.5cm]{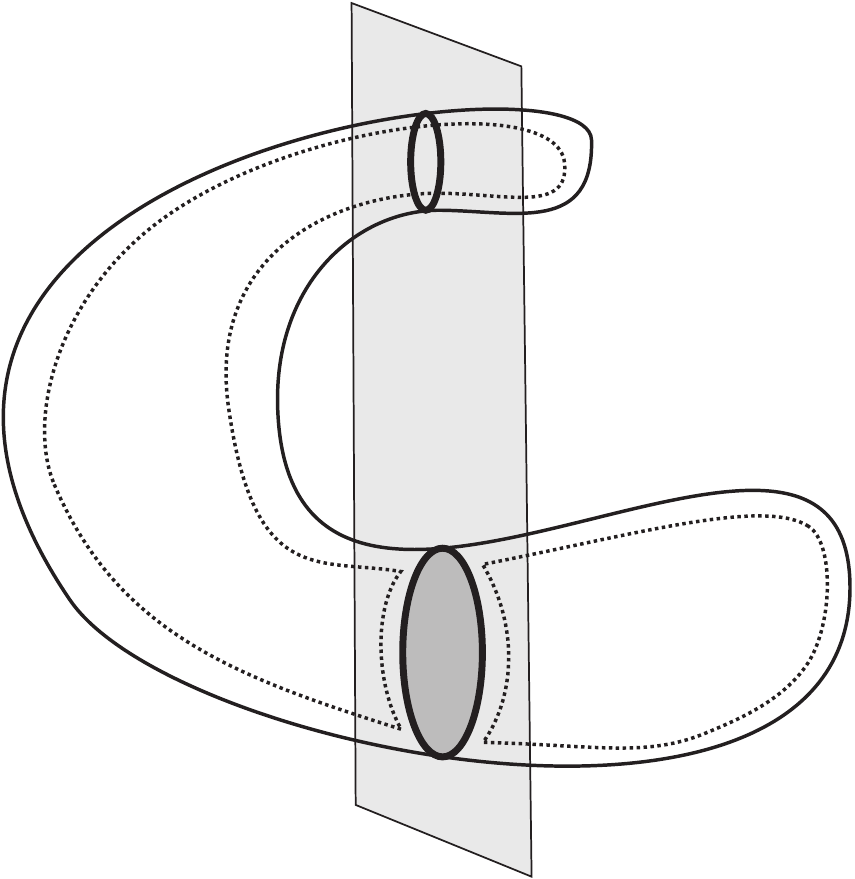}}
  \caption{Removing intersections}
  \label{maintrick}
\end{figure}

\begin{theorem} \label{one} For any 3--manifold $M$, the root $R(M)$ under
spherical reduction exists and is unique up to homeomorphism and removal
of spherical  components.
\end{theorem}
\begin{proof} This follows from \fullref{main}.
\end{proof}

Note that $R(M)$ is the disjoint union of the irreducible factors of $M$,
ie,~of the prime factors other than $S^2{\times} S^1$ or
$S^2{\tilde \times} S^1$. So the root appears  more natural than the
Kneser--Milnor decomposition and the same is true for other applications of
\fullref{main},
see the following sections. In order to get the full
version of the prime decomposition theorem, additional efforts are needed.
The advantage of this two-step method is that the first step (existence
and uniqueness of roots) is more or less standard, while all individual
features appear only at the second step.

In our situation, note that the number $n$ of factors $S^2{\times}S^1$
is determined by $M$.  Let $M$ be orientable.  Denote by $\Sigma(M)$
the subgroup $\Sigma(M)\subset H_2(M;Z_2)$ consisting of all spherical
elements (that is, of elements of the form $f_\ast(\mu)$, where $\mu$
is the generator of $H_2(S^2;Z_2)$ and $f\colon
S^2\to M$ is a map). Let us show that $n$ coincides with the rank
$r(M)$ of $\Sigma(M)$ over $Z_2$. This follows from the following
observations.
\begin{enumerate}
\item Spherical reductions along separating spheres
preserve $\Sigma(M)$ and $r(M)$ while any spherical reduction along a
non-separating sphere decreases $r(M)$ by one and kills one $S^2{\times}S^1$
factor.
\item After performing all possible spherical reductions,
we get a 3--manifold $M'$, which is aspherical (according to the Sphere
Theorem) and thus has $r(M')=0$.
\end{enumerate}

We thus easily deduce the Kneser--Milnor theorem from \fullref{one}:
$M$ is the connected sum of the connected components of $R(M)$ plus $n$
factors $S^2{\times}S^1$. The  factors   are determined uniquely up to
homeomorphism.

\section{Decomposition into boundary connected sum}

\label{D-roots} By a {\itshape disc reduction} of a 3--manifold $M$ we mean
cutting $M$ along a proper disc $D$. If $\p D$ bounds a disc in $\p M$,
then the reduction is called {\em trivial}. It makes little sense to
consider disc reductions separately from spherical ones. Usually one uses
them together or considers only irreducible manifolds. We start with the
first approach. By an $SD$--reduction we will mean a spherical or a disc
reduction.

As in \fullref{S-roots},
we begin by introducing an oriented
graph $\Gamma$. The set of vertices of $\Gamma$ is defined to be the set
of all compact 3--manifolds, but considered up to homeomorphism and
removing connected components homeomorphic to $S^3$ or $D^3$. Two vertices
$M_1,M_2$ are joined by an edge oriented from $M_1$ to $M_2$ if $M_2$ can
be obtained from $M_1$ by a non-trivial $SD$--reduction.

\begin{lemma} \label{edgequivD} $\Gamma$ possesses properties (CF) and
 (EE).
\end{lemma}
\begin{proof}
Let $M$ be a 3--manifold. We assign to it two integer numbers
$g^{(2)}(\partial M)$ and $s(M)$. The first number is equal to  $\sum
g^2(F)$, where $g(F)$ is the genus of a component $F \subset \partial M$,
and the sum is taken over all components  of $\partial M$. Note that
{\itshape any} non-trivial disc reduction  lowers $g^{(2)}(\partial M)$
while $g^{(1)}(\partial M)$, the sum of the genera of the components
$F\subset \partial M$, stays unchanged when $\partial D$ is  separating.
The second number $s(M)$ is the maximal number of non-trivial spherical
reductions on $M$ introduced in the previous section, where we denoted it
by $c(M)$. Now we introduce the complexity function $c\colon V(\Gamma) \to
\N^2$ by letting $c(M)=(g^{(2)}(\partial M), s(M))$, where the pairs are
considered in lexicographical order.
 It is clear that $c$ is compatible with the orientation of $\Gamma$.

For proving property (EE), we apply the same method as in the proof of
Lemma 2. Let $F_e,F_d$ be two surfaces such that each of them is either a
non-trivial sphere or a non-trivial disc. We will use the same induction on
the number $m$ of curves in $F_e\cap F_d$.

{\bf Base of induction}\qua Let $m=0$ and let $M_e,M_d$ be the manifolds
obtained by reducing $M$ along $F_e$ and $F_d$, respectively.   We wish to
prove that $M_e$ and $M_d$ have a common root. Let $N$ be obtained by
reducing $M_e$ along $F_d$ as well as by reducing $M_d$ along $F_e$.

Suppose that either both $F_e\subset M_d$ and $F_d\subset M_e$ are
non-trivial, or both are spheres, or one of them is a sphere and the other
is a disc (in the latter case the disc is automatically non-trivial, since
spherical reductions do not affect the property of a disc to be
non-trivial). Then any root of $N$ is a common root of $M_e,M_d$ (the same
proof as in Lemma 2 does work). That covers all cases when at least one of
the surfaces is a sphere.

Suppose that both surfaces $F_e,F_d$   are discs such that
      $F_e$ is trivial in $M_d$.
Then the disc
    $D\subset \partial M_d$  bounded by
      $\partial F_e$   must contain at least one
      of the two copies $F_d^+,F_d^-$ of $F_d$   appeared under the
      cut.  Let $S\subset \Int M_d$ be a
    sphere which runs along $F_e$ and  $D $    such
     that $S\cap F_e=\emptyset$. The last condition guarantees us that
     $S$ may be considered as a sphere sitting  in $M$ as well as
     in $M_e$.  Now there are two possibilities.
\begin{enumerate}

\item
$D$ contains only one copy (say, $F_d^+$) of $F_d$
 (in this case the disc
 $F_d\subset M_e$ also is trivial).  Then the
 spherical reductions of $M_e$ and $M_d$ along $S$ produce the
 same manifold $N'$. So any  root of $N'$ is a common
 root of $M_e$ and
 $M_d$.

\item $D$ contains both copies   $F_d^+$ and $F_d^-$ of   $F_d$
 (in this case the disc
 $F_d\subset M_e$  is non-trivial). Note that $S$ survives the reduction
     along $F_e$ and thus may be considered as a sphere in $N$.
  Denote by $N' $ the manifold obtained from $M_d$ by spherical
  reduction along $S$. It is easy to see that the
  spherical reduction of  $N$ along $S$ produces a manifold homomorphic to the
   disjoint union of
  $N'$ and a 3--ball.  Since 3--balls are neglected, any root of
  $N'$ is a common root of  $M_e$ and
 $M_d$.
$$\xymatrix@!@-25pt{&&M\ar[ddll]_{F_e}\ar[ddrr]^{F_d}\\
          \\
          M_e\ar[dr]_{F_d}&&&&M_d\ar[ddll]^S\\
          &N\ar[dr]_S\\
          &&N'}$$
\end{enumerate}

The inductive step is performed exactly in the same way as in the proof of
Lemma 2.  The only difference is that in addition to an innermost circle
argument we use an outermost arc argument for decreasing the number of
arcs in the intersection of discs.\end{proof}

\begin{theorem} \label{two} For any 3--manifold $M$
the $SD$--root $R(M)$ exists and is unique up to homeomorphism and removal
of spherical and ball connected components.
\end{theorem}
\begin{proof} This follows from
\fullref{main}
and \fullref{edgequivD}.
\end{proof}

If a 3--manifold $M$ is irreducible, then one can sharpen
\fullref{two}
by considering only $D$--reductions. Any such disc
reduction can be realized by removing an open regular neighborhood of $D$
in $M$ and getting a submanifold of $M$. So any $D$--root of $M$ is
contained in $M$.

\begin{theorem} \label{twoprime} For any irreducible 3--manifold $M$,
the $D$--root $R(M)$ exists and is unique up to isotopy and removal
of   ball connected components.
\end{theorem}
\begin{proof} Given $M$, we begin by introducing
an oriented graph $\Gamma=\Gamma (M)$. The set of vertices of $\Gamma$ is
defined to be the set of all compact 3--submanifolds of $M$,   considered
up to isotopy and removing connected components homeomorphic to  $D^3$.
Two vertices $Q_1,Q_2$ are joined by an edge oriented from $Q_1$ to $Q_2$
if $Q_2$ can be obtained from $Q_1$ by a non-trivial $D$--reduction.
Property  (CF)  for $\Gamma$ is evident: one can take the complexity
function $c\colon V(\Gamma) \to \N$ by letting $c(Q)=g^{(2)}(\partial Q)$.
Property (EE) can be proved exactly in the same way as in the proof of
\fullref{edgequivD}.
The only difference is that cutting along a trivial
disc produces the same manifold plus a ball (because of irreducibility)
and thus preserves the corresponding vertex of $\Gamma$. Therefore, the
conclusion of the theorem follows from \fullref{main}.
\end{proof}

\begin{remark}   This
gives another proof of Bonahon's Theorem  that the
\emph{characteristic compression body} (which can be defined as
the complement to a $D$--root $R(M)\subset \Int M$) is unique up to
isotopy, see~\cite{Bo}.
\end{remark}

In the previous section we have mentioned that there is only one closed
orientable 3--manifold which is reducible and prime. On the contrary, the
number of boundary reducible 3--manifolds which are $\partial$--prime (ie,
prime with respect to boundary connected sums)  is infinite. For example,
if we take an irreducible boundary irreducible 3--manifold with $n>1$
boundary components $C_1,\ldots C_n$ and join each $C_i, \: i<n,$ with
$C_{i+1}$ by a solid tube, we obtain a boundary reducible $\partial$--prime
manifold. (In analogy to the closed case, the $\partial$--reductions take
place along non-separating discs). Nevertheless, any compact 3--manifold
has a unique decomposition into a boundary connected sum of
$\partial$--prime factors. This theorem was first proved by A~Swarup
\cite{swar}. Below we show that for irreducible 3--manifolds it can easily
be deduced from \fullref{twoprime}.
\begin{theorem} \label{twom}
Any  irreducible  3--manifold $M$ is a boundary connected sum of
$\partial$--prime factors. The  factors are determined uniquely up to
homeomorphism and -- in the non-orientable case -- replacement of the
direct product $D^2{\times}S^1 $ by the skew product $D^2 {\tilde \times
S^1} $ and vice versa.
\end{theorem}

\begin{proof} Existence of a boundary prime decomposition is
evident. Let us prove the uniqueness. Recall that any disc reduction can
be realized by cutting out  a regular neighborhood of a proper disc. We
are allowed also to cast out 3--balls. It follows that we may think of the
root $N=R(M)$ of $M$ (which is well defined by \fullref{twoprime})
as sitting in $M$ such that $M$
can be obtained from $N$ by adding disjoint 3--balls and attaching handles
of index one. Therefore $M$ can be presented as $M=N \cup H$ such that $H$
is a union of disjoint handlebodies and $N\cap H $ is a collection of
discs on $\p N$.

One may easily achieve  that the intersection of any connected
component $N_i$ with $H $ consists of no more than one disc, see
\fullref{yourpict}.

\begin{figure}[h]
\labellist
\small
\pinlabel {$H$} [l] at 198 755
\pinlabel {$H$} [l] at 430 760
\pinlabel {$N_i$} at 148 673
\pinlabel {$N_i$} at 379 666
\endlabellist
\centerline{\includegraphics[height=3cm]{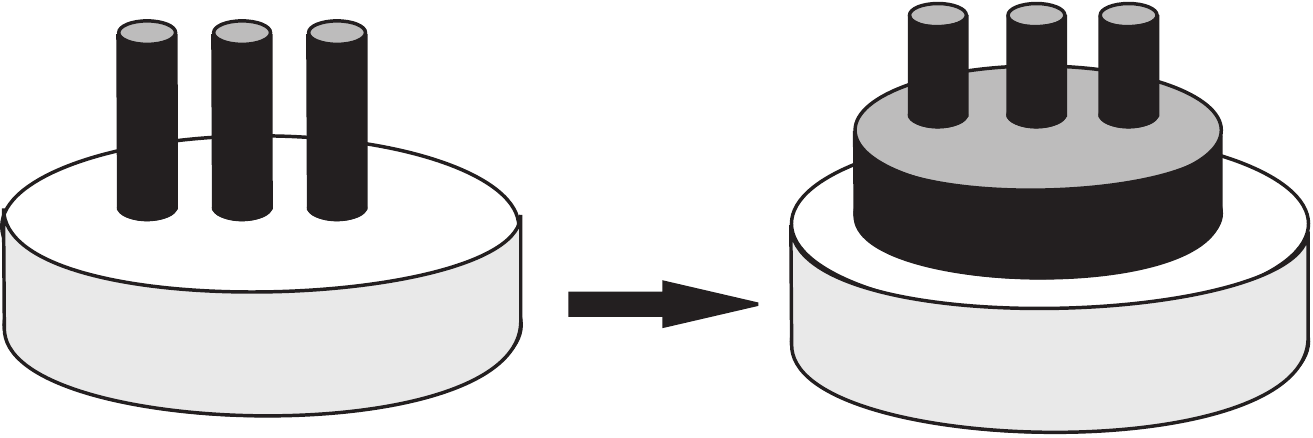}}
  \caption{The intersection of $N_i$ with $H$ consists of no more than one disc.}
  \label{yourpict}
\end{figure}
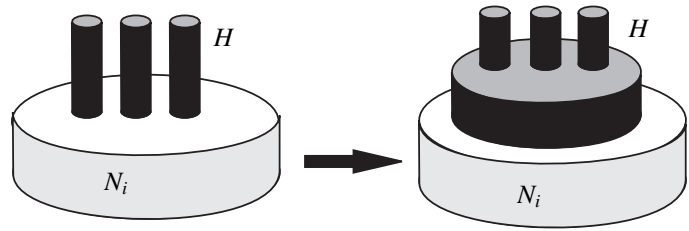

Moreover, the presentation $M=N \cup H$ as above is unique up to
isotopy.

We now wish to cut $H$ to get from the root to a prime
decomposition, see \fullref{ddeco1}: 

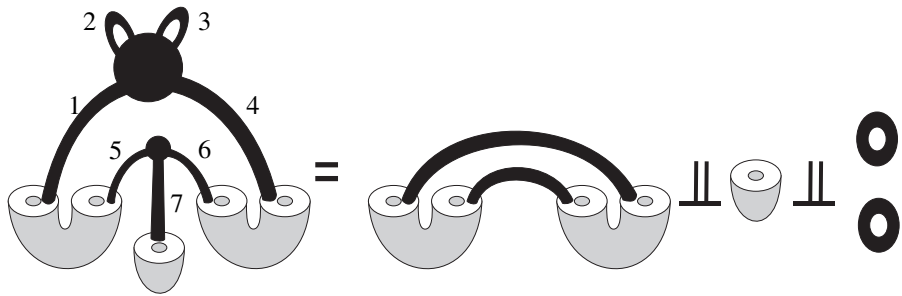
\begin{figure}[h]
\labellist
\small
\pinlabel {$1$} [l] at 70 710
\pinlabel {$2$} [l] at 78 755
\pinlabel {$3$} [l] at 140 755
\pinlabel {$4$} [l] at 166 710
\pinlabel {$5$} [l] at 92 686
\pinlabel {$6$} [l] at 140 686
\pinlabel {$7$} [l] at 125 657
\endlabellist
\centerline{\includegraphics[width=0.9\hsize]{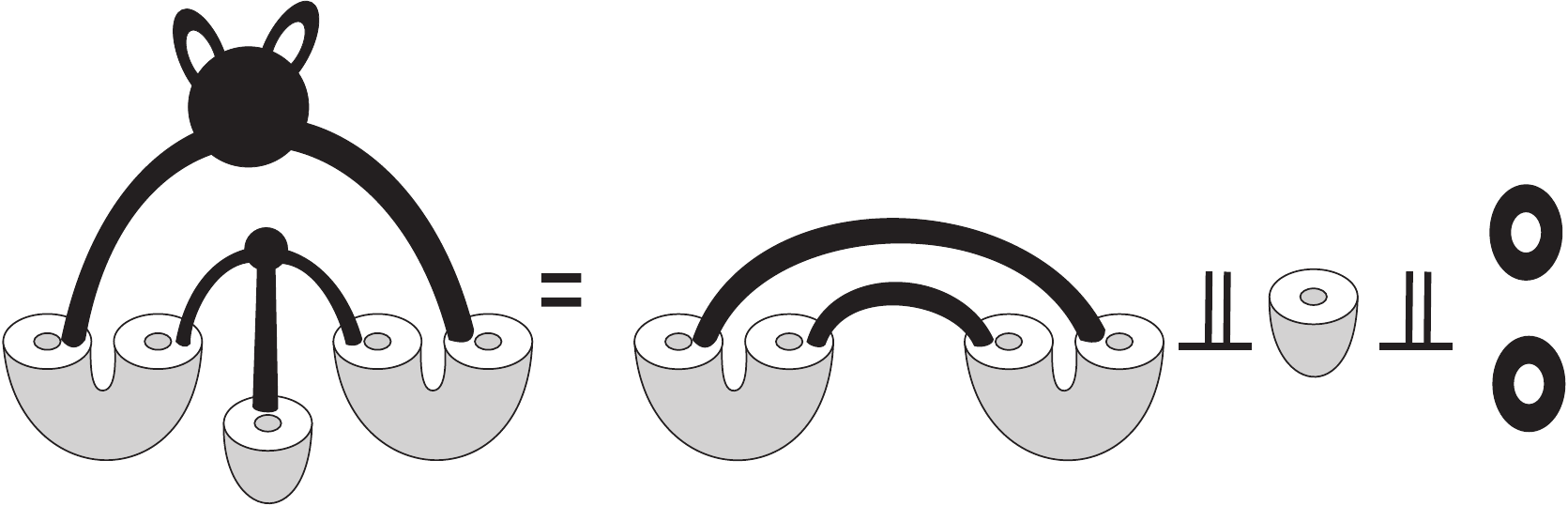}}
  \caption{$H$ consists of two balls and seven handles.}
  \label{ddeco1}
\end{figure}

Temporarily shrink each connected component of $N$ to a red point
and each component of $H$ to a core with exactly one non-red (say: green)
vertex. Call the resulting ancillary graph $G$.

$G$ \emph{admits cutting} if there is a  graph $\hat{G}$  (again
with vertices coloured green and red) such that $\hat{G}$ has one more
component than $G$ and $G=\hat{G}/x_1 \sim x_2$ for some green vertices
$x_1, x_2$ of $\hat{G}$.

Note that the cutting loci (and thus also the outcome of a maximal
cutting procedure) are uniquely determined by the equivalence relation
generated by the following rule: edges outgoing from the green vertex
$x$ are elementary equivalent if they lie in the same connected
component of $G-\{x\}$.

Let $M'$ be obtained  from $M$ by cutting along discs corresponding to
such cuts of $G$.

It follows from the construction that the connected components of the
resulting new manifold $M'$ are the $\p $--prime factors of $M$.
\end{proof}

\section{Roots of knotted graphs}
Now we will consider pairs of the type $(M,G)$, where $M$ is a compact
3--manifold and $G$ is an arbitrary graph (compact one-dimensional
polyhedron) in $M$. Recall that a 2--sphere $S\subset M$ is {\em in general
position} (with respect to $G$), if it does not pass through vertices of
$G$ and intersects edges transversely. It is {\em clean} if its
intersection with $G$ is empty.

\begin{definition} \label{df:compression} Let   $S$ be a general
position   sphere in $(M,G)$. Then the {\em reduction } (or {\em
compression}) of $(M,G)$ along $S$ consists in cutting $(M,G)$
along
 $S$ and taking disjoint cones over $(S_{\pm}, S_{\pm}\cap G)$, where
 $S_{\pm}$ are the two copies of $S$  appearing under the cut.
\end{definition}

Equivalently, the reduction along $S$ can be described as
compressing $S$ to a point and cutting the resulting singular
manifold along that point.

If   $(M',G')$ is obtained from $(M,G)$ by reduction along
 $S$, we write $(M',G')=(M_S,G_S)$.
 The two cone points in $M_S$ are called {\em stars}.
 They lie in $G_S$ if and only if
 $S\cap G\neq \emptyset$.

It makes little sense to consider all possible spherical
reductions, since then there would be no chance to get existence
and uniqueness of a root, see below. In order to describe
allowable reductions of the pair $(M,G)$, we introduce two
properties of spheres in $(M,G)$.

\begin{definition} \label{df:admis} Let $S$ be a 2--sphere in
$(M,G)$. Then $S$ is called
\begin{enumerate}
\item[(1)] {\em compressible} if there is a disc $D\subset M$ such
that $D\cap S=\p D$,
 $D\cap G=\emptyset$, and each of the two discs bounded by $\p D$ on $S$
 intersects $G$; otherwise $S$ is {\em incompressible};

 \item[(2)] {\em admissible}  if $S\cap G$ consists of no more than three
transverse crossing points.

  \end{enumerate}
\end{definition}

The following examples show that compressions along compressible and
inadmissible spheres may produce different roots.

\begin{Example}  Take the   knot $k$ in $M=S^2{\times}S^1$
shown in \fullref{counter}. 

\begin{figure}[h]
\labellist
\large
\pinlabel {$ \bigcup $} [l] at 424 686
\endlabellist
\centerline{\includegraphics[height=2.5cm]{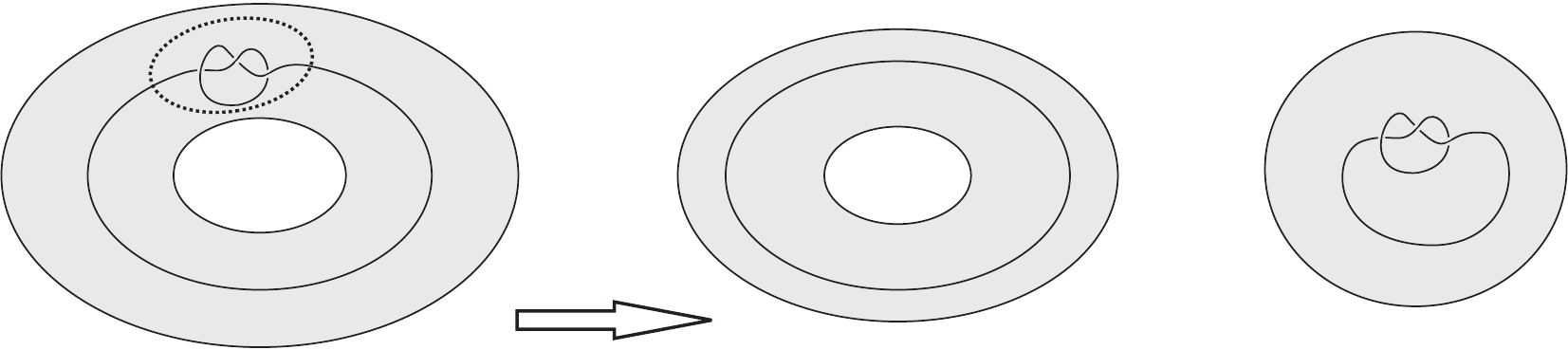}}
  \caption{Example of a spherical reduction along a compressible sphere}
  \label{counter}
\end{figure}

$k$ mainly follows $\{ \ast \}{\times}S^1$, but has a trefoil in it. If we
allowed to perform reduction along the dotted sphere $S$ (which is
compressible!), then $(M,k)$ would split off a $3$--sphere containing the
trefoil knot. But note that $k$ is in fact equivalent to $\{ \ast \}
{\times}S^1$. Indeed, by deforming some little arc of the trefoil all the
way across $ S^2   {\times} \{ \ast \}$, we can change an overcrossing to an
undercrossing so that the knot $k$ comes undone. Thus, $(M,k)$ would equal
$(M,k)$ plus a non-trivial summand and there would be no hope for
uniqueness of the root.\end{Example}

\begin{Example}  Let $(M,G)$ be the standard circle with two parallel
chords in $S^3$. As we see from \fullref{five}, compressions of $(M,G)$
along two different spheres (admissible and non-admissible) produce two
different roots.\end{Example}

Note that the existence of an admissible compressible sphere implies the
existence of either a separating point of $G$ or an $(S^2{\times}
S^1)$--summand of $M$.

\begin{figure}[ht]
\centerline{\includegraphics[height=5cm]{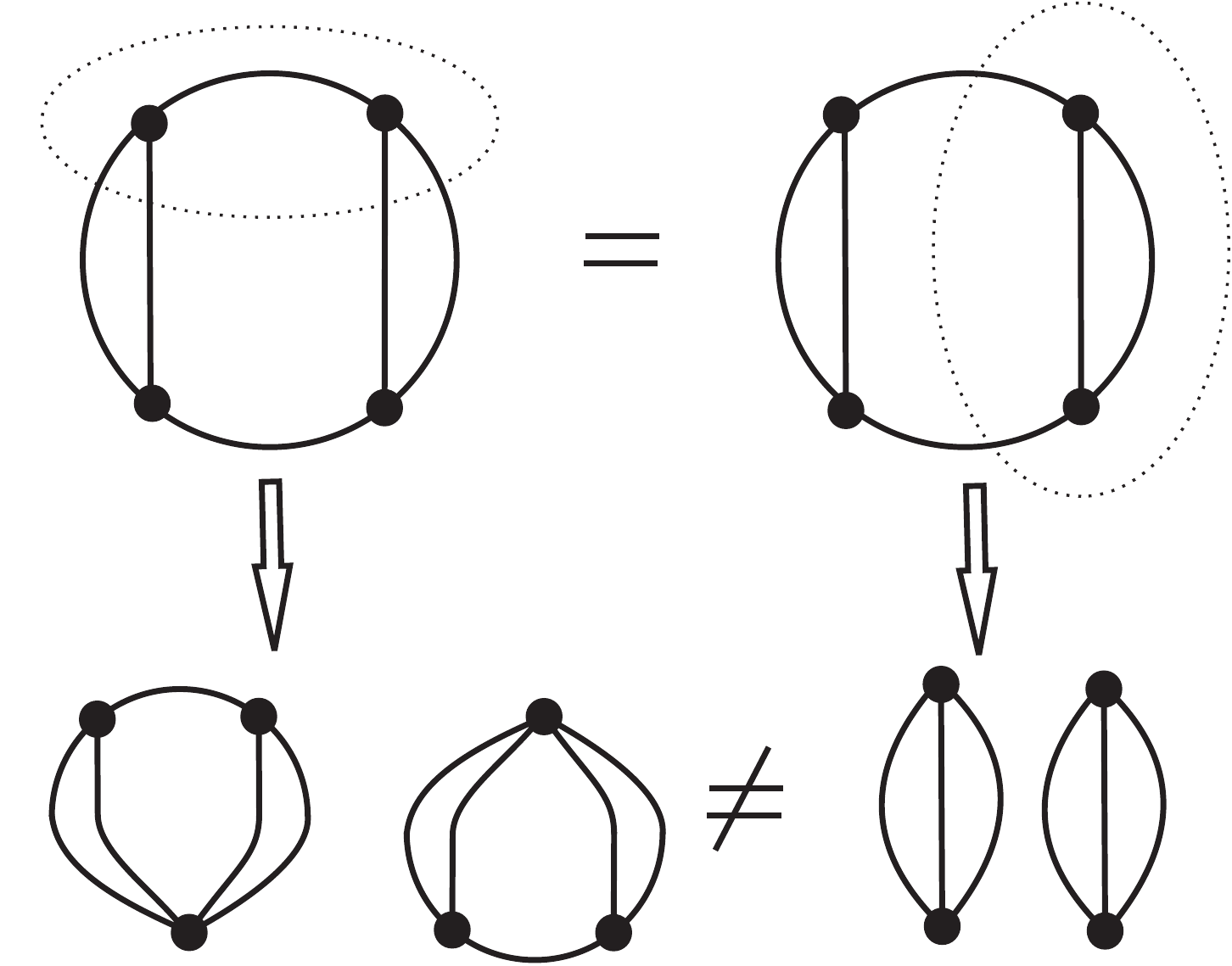}}
  \caption{Example of a spherical reduction along an inadmissible sphere}
  \label{five}
\end{figure}

\begin{definition} \label{df:trivial} A sphere $S$ in $(M,G)$ is called {\em
trivial} if it bounds a ball $V\subset M$ such that the pair $(V, V\cap
G)$ is homeomorphic to the pair $(\Con (S^2), \Con (X))$, where $X\subset
S^2$ consists of $\leq 3$ points and $\Con$ is the cone.
An incompressible admissible non-trivial sphere  is called {\em essential}.
\end{definition}

Note that reduction of $(M,G)$ along   a trivial sphere produces a
homeomorphic copy of $(M,G)$ and a {\em trivial} pair $(S^3,G)$, where $G$
is either empty, or a simple arc, or an unknotted circle, or an unknotted
(ie contained in a disc) theta--curve.

Let $G$ be a knotted graph in a 3--manifold $M$. Following the same lines
as in the previous sections, we begin by introducing an oriented graph
$\Gamma $. The set $V(\Gamma)$ of vertices   is defined to be the set of
all pairs $(M,G)$, considered up to homeomorphism of pairs and removing
connected components homeomorphic to trivial pairs. Two vertices
$(M_1,G_1), (M_2,G_2)$ are joined by an edge oriented from $(M_1,G_1)$ to
$(M_2,G_2)$ if $(M_2,G_2)$ can be obtained from $(M_1,G_1)$ by reduction
along some  essential  sphere.

We need an analogue of \fullref{boundg}
(Kneser's Lemma) for manifolds with knotted graphs.

\begin{lemma} \label{boundgraph} Suppose that $(M,G)$ contains no clean essential
 spheres, ie, that the manifold $M\setminus G$ is irreducible.
   Then there is a constant $C_1$ depending only on
$(M,G)$ such that any sequence of   reductions along essential
spheres consists of no more than $C_1$ moves.
\end{lemma}
\begin{proof} We choose a triangulation $T$ of  $(M,G)$ such
that $G$ is the union of (some of the) edges and vertices of $T$. Let
$C_1=10t$, where $t$ is the number of tetrahedra in $T$. Consider a
sequence $S_1,\ldots, S_n\subset (M,G)$ of $n>C_1$ disjoint spheres such
that each sphere $S_k$ is essential in the pair $(M_k,G_k)$ obtained by
reducing $(M_0,G_0)=(M,G)$ along $S_1,\ldots, S_{k-1}$.
 It is easy to see that
the spheres   are essential in $(M,G)$ and not parallel one to
another.

We claim that the union $F=S_1\cup \cdots \cup S_n$  can be shifted into
normal position by an isotopy of the pair $(M,G)$. To prove that, we
adjust to our situation two types of moves for the standard normalization
procedure of arbitrary incompressible surface in an irreducible
3--manifold.

{\bf Tube compression}\qua Let $D$ be a disc in $(M,G)$ such that
  $D\cap F=\p D$ and $D$ does not intersect the edges of $T$
  (in practice $D$ lies either in a face or in the
  interior of a tetrahedron). Since $F$
  is a collection of incompressible spheres, compression along
  $D$ produces a copy $F'$ of $F$ and a clean sphere $S'$. By
  irreducibility of $M\setminus G$, $S'$   bounds a clean ball,
    which helps us to construct an isotopy  of $(M,G)$ taking
  $F$ to $F'$.

{\bf Isotopy through an edge}\qua Let $D$ be a disc in a tetrahedron $\Delta$
such that $D\cap F$ is an arc in $\p D$ and $\p D$ intersects an edge $e$
of $\Delta$ along the  complementary arc of $\p D$.  If   $e$ were
in $G$, by incompressibility there could be no further intersection of $G$
and the component of $F$ containing the arc of $\p D$. As each clean
sphere bounds a ball this component would be a trivial sphere in $(M,G)$.
Thus, $e$ is not in $G$, so we can use $D$ to construct an isotopy of
$(M,G)$ removing two points in $F \cap e$.

Performing those moves as long as possible, we transform $F$ into
normal position.

 Now we notice that the normal surface $F$
decomposes $M$ into pieces (called {\em chambers}) and  crosses each
tetrahedron of $T$ along triangle and quadrilateral pieces (called
{\em patches}). Let us call a patch
  {\em black}, if it   does not lie between two parallel patches
  of the same type. Each tetrahedron contains at most  $10$ black
  patches: at most 8 triangle patches and at most 2 quadrilateral ones.
   Since $n>10t$, at least one of the spheres  is
  {\em white}, ie contains no black patches. Let $C$ be a
  chamber such that $\partial C$ contains a white sphere and a
  non-white sphere.
 Then $C$ crosses each tetrahedron along some number of
  prisms of the type $P{\times} I$, where $P$ is a triangle or a
  quadrilateral.

Since the patches $P\times \{ 0,1\} $ belong to different spheres, $C$
is $S^2\times I$.

  This contradicts our assumption that the spheres
  be not parallel.
\end{proof}

\begin{lemma} \label{new} Let a pair $(M_S,G_S)$ be obtained from
a pair   $(M,G)$ by reduction along an incompressible sphere $S$
such that $S\cap G\neq \emptyset$. Then $M\setminus G$ is
reducible if and only if so is $M_S\setminus G_S$.
\end{lemma}
\begin{proof} Let $S_1$ be a clean essential sphere in $M$. Suppose
that $S\cap S_1\ne \emptyset$. Using an innermost circle argument,
we find a disc $D_1\subset S_1$ such that $D_1\cap S=\partial
D_1$. Since $S$ is incompressible, $\partial D_1$ bounds a clean
disc $D\subset S$. Now we may construct a new clean sphere
$S_1'\subset M$ such that $S'_1\cap S$ consists of a smaller
number of circles than $S_1\cap S$. Indeed, if the sphere $D\cup
D_1$ is essential, we take $S_1'$ to be a copy of $D\cup D_1$
shifted away from $S$. If $D\cup D_1$ bounds a clean ball, we use
this ball for constructing an isotopy of $D_1$ to the other side
of $S$. This isotopy takes $S_1$ to $S_1'$.

Doing so for as long as possible, we get a clean essential sphere
$S_1\subset M$ such that $S\cap S_1=\emptyset$. It follows that
then $S_1$, considered as a sphere in $M_S$, is essential.

The proof of the lemma in the reverse direction is evident.
\end{proof}

\begin{lemma} \label{bound2} For any pair  $(M,G)$ there exists a
constant $C $ such that any sequence of reductions along essential spheres
consists of no more than $C$ moves.
\end{lemma}
\begin{proof} Let $S_1,\ldots S_n\subset (M,G)$ be
such a sequence of   essential spheres. For any $k, 1\leq k\leq n
,$ we denote by $(M_k,G_k)$ the pair obtained from
$(M_0,G_0)=(M,G)$ by reductions along the spheres $S_1, \ldots,
S_{k}$. We may assume that the last pair $(M_n,G_n)$  admits no
further reductions along clean essential spheres. Otherwise we
extend the sequence of reductions by new reductions along clean
essential spheres until we get a pair with irreducible graph
complement.
 Denote also by
$(M'_k,G'_k)$ the pair obtained from $(M_k,G_k)$ by additional reductions
along all remaining clean spheres from the sequence $S_{1},\ldots S_n$.
Note that $(M'_n,G'_n)$ is obtained from $(M'_k,G'_k)$ by reductions along
dirty spheres and that $M'_n\setminus G'_n$ is irreducible. By
\fullref{new}
$ M'_k\setminus G'_k$ also is irreducible and hence
contains no clean essential spheres.

It is convenient to locate the set $X$ of {\em clean stars} (the images of
the cone points under reductions of all  clean essential spheres from
$S_1,\ldots S_n$). Then $X$ consists of no more than $2C_0$ points, where
$C_0=C_0(M,G)$ is the constant from \fullref{boundg}
for a compact
3--manifold whose interior is $M\setminus G$. We may think of $X$ as being
contained in all $(M'_k,G'_k)$.

Let us decompose the set $S_1,\ldots S_n$ into three subsets
$U,V,W$ as follows:
\begin{enumerate}
\item $S_k\in U$ if  $S_k$ is clean.

\item $S_k\in V$ if  $S_k$, considered as a sphere in
$(M'_{k-1},G'_{k-1})$, is an essential sphere (necessarily dirty).

\item $S_k\in W$ if $S_k$ is a trivial dirty sphere in
$(M'_{k-1},G'_{k-1})$.
\end{enumerate}

Now we estimate the numbers $\# U,\# V,\# W$ of spheres in $U,V,W$.  Of
course, $\# U\leq  C_0$ and $\# V\leq  C_1$, where $C_0$ is as above and
$C_1=C_1(M'_0,G'_0)$ is the constant from \fullref{boundgraph}.
Let us
prove that $\# W\leq 2C_0$. Indeed, the reduction along each sphere
$S_k\in W$ transforms $(M'_{k-1},G'_{k-1})$ into a copy of
$(M'_{k-1},G'_{k-1})$ and a trivial pair $(S^3_{k-1},\Gamma_{k-1})$
containing at least one
  clean star.   Since no star can appear in two different trivial
  pairs and since the total number of clean
stars does not exceed $2C_0$, we get $\# W\leq 2C_0$. Combining
these estimates, we get $n\leq C=3C_0+C_1$.
\end{proof}

\begin{lemma} \label{edgequivKG} The graph $\Gamma=\Gamma (M,G)$
possesses properties (CF) and
 (EE).
\end{lemma}

\begin{proof}
 We define the  complexity function $c\colon V(\Gamma) \to
\N_0$ just as in   the proof of \fullref{comfun}:
$c(M,G)$ is
the maximal number of   reductions in any sequence of essential
spherical reductions of $(M,G)$. By \fullref{bound2},
this
is well defined. Evidently, $c$ is compatible with the
orientation.

 The proof of property (EE) is similar to the proof of the same
 property for the case $G=\emptyset$, see \fullref{edgequiv}.
Let  $S_e ,S_d$ be two essential spheres in $(M,G)$ corresponding
to edges   $e, d$  of $\Gamma$. We prove
 the equivalence $e\sim d$ by induction on the number
 $m=\# (S_e\cap S_d)$ of curves
 in the intersection assuming that the  spheres have been
   shifted by isotopy of $(M,G)$ so that $m$ is minimal.
   The base of the
 induction, when $ m=0$, is evident, since
 reductions along disjoint spheres commute
 and thus, just as in the proof of \fullref{edgequiv}.
  produce knotted graphs having a common root.

Let us prove the inductive step.  Suppose that   $S_d$  contains a
disc $D$ such   $D\cap S_e=\p D$. Assume that $D$ is clean, ie
$D\cap G=\emptyset$.  Then we compress $S_e$ along $D$. We get two
spheres $S' ,S'' $, each disjoint with $S_e$ and intersecting
$S_d$ in a smaller number of circles. Since $S_e$ is
incompressible and $m$ is minimal, at least one of them (say $S'$)
must be clean and
  essential. As $S_e\sim S'$ and
$S'\sim S_d$ by the inductive assumption, we get $S_e\sim S_d$.

Now we may assume that $S_d$ (and, by symmetry, $S_e$) contain  no
innermost clean discs. Since $S_d$ contains at least two innermost
discs and $\# (S_d\cap G)\leq 3$, there is an innermost disc
$D\subset S_d$ crossing $G$ at exactly one point. Its boundary
decomposes $S_e$ into two discs $D',D''$, both crossing $G$. Since
$\# (S_e\cap G)\leq 3$, at least one of them (let $D'$) crosses
$G$ at exactly one point.  Then the sphere $S_e'=D\cup D'$ is
admissible, incompressible, and non-trivial. Moreover, it is
actually disjoint with $S_e$ and crosses $S_d$ in less than $m$
circles. Using the inductive assumption, we get   $S_e\sim S_d$
again.\end{proof}

\begin{theorem} \label{three} For any knotted graph $(M,G)$
the  root $R(M,G)$ exists and is unique up to homeomorphism and removal of
trivial components.
\end{theorem}
\begin{proof} This follows from
\fullref{main}
and \fullref{edgequivKG}.
\end{proof}

 According to our definition of  the graph $\Gamma$ for the case of
 knotted graphs,
 its vertices and hence roots of knotted graphs  are defined only
 modulo removing trivial pairs.  One of the advantages of
   roots introduced in that manner is the  flexibility
 of their construction:   each next reduction    can be performed
  along {\em any} essential sphere.
  We pay for that by the  non-uniqueness:
     roots of   $(M,G)$ can   differ by their trivial
  connected components.  This is natural, but might seem to
 be inconvenient. We improve this by introducing {\em efficient roots},
 which are free from that shortcoming
 (the idea is borrowed from Petronio~\cite{petr}).

\begin{definition} A system ${\cal S}=S_1\cup  \cdots \cup S_n$
 of disjoint incompressible spheres
  in $(M,G)$ is called {\em efficient} if the following holds:
  \begin{enumerate}
  \item[(1)] reductions along all the spheres give a root of $(M,G)$;
  \item[(2)] any sphere $S_k, 1\leq k\leq n,$ is essential in the
  pair $(M_{{\cal S}\setminus S_k}, G_{{\cal S}\setminus S_k})$ obtained
   from $(M,G)$ by reductions along all
  spheres $S_i,1\leq i\leq n,$ except $S_k$.
\end{enumerate}
  \end{definition}

  Efficient systems certainly exist; to get one, one may
  construct a system
  satisfying (1) and merely  throw
   away one after another  all spheres not satisfying (2).
Having an efficient system, one can get another one by the
following
 moves.
\begin{enumerate}
\item  Let $a\subset (M,G)$ be a clean simple
 arc which joins a sphere $S_i$ with a clean
 sphere $S_j,  i\neq j,$ and has no  common points  with  $\cal S$ except
 its ends. Then the  boundary $\p N$ of a regular neighborhood
 $N(S_i\cup a\cup S_j)$ consists of a copy of $S_i$, a copy of $S_j$, and the
  (interior) connected sum $S_i\# S_j$ of $S_i$ and $S_j$.  The move consists
 in replacing $S_i$ by $S_i\# S_j$.

 \item The same, but with the following modifications:
 \begin{enumerate}

 \item[i)] $a$ is a simple subarc of $G$ such that all vertices of $G$
 contained in $a$ have valence two, and
 \item[ii)] $S_j$ crosses $G$ in two   points.
\end{enumerate}
\end{enumerate}
 Both moves are called {\em spherical slidings}, see \fullref{slide}. 
Note that spherical slidings do not affect the corresponding root.

\begin{figure}[h]
\labellist
\small
\pinlabel {$S_i$} [l] at 100 685
\pinlabel {$S_j$} [l] at 246 685
\pinlabel {$S_i\# S_j$} [l] at 160 761
\endlabellist
\centerline{\includegraphics[height=4cm]{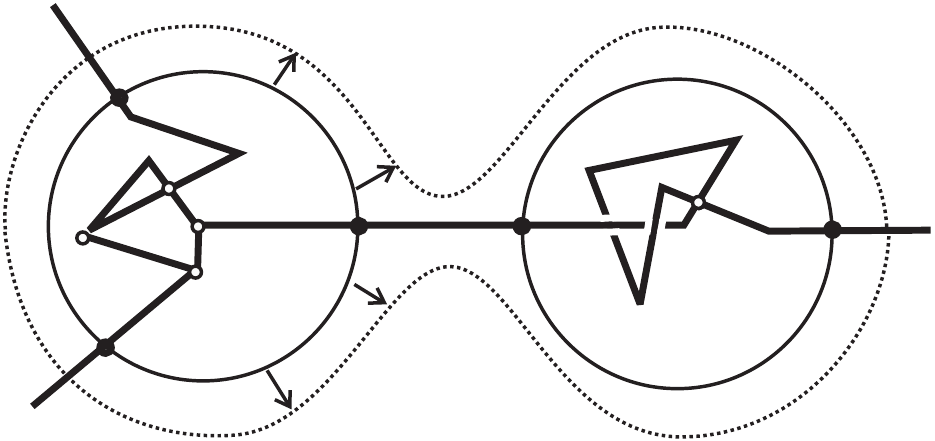}}
  \caption{Spherical sliding}
  \label{slide}
\end{figure}

\begin{definition} Two efficient system in $(M,G)$ are {\em equivalent} if
one system can be transformed into the other by a sequence of
spherical slidings and an isotopy of $(M,G)$.
\end{definition}

The following theorem can be   proved by a  modification of the
proof of property (EE) in \fullref{edgequivKG}.
See~\cite{mpim2}
for details.

\begin{theorem} \label{effic} Any two efficient systems in $(M,G)$ are
 equivalent.\endproof
 \end{theorem}

\begin{definition} \label{effroot} A root of $(M,G)$ is {\em efficient}, if
it can be obtained by reducing $(M,G)$ along all spheres of an
efficient system.
\end{definition}

 \begin{theorem}  \label{main1} For any     $(M,G)$   the efficient
 root   exists and is unique up to homeomorphism.
\end{theorem}
\begin{proof} This is evident, since spherical slidings of an efficient
system do not affect the corresponding root.
\end{proof}

\begin{remark}
\fullref{main1}
easily implies the Schubert
Theorem on the uniqueness of prime knot decomposition in $S^3$ as well as
the corresponding theorem for knots in any irreducible 3--manifold.
Indeed, the connected components of the efficient root of $(M,K)$ are
exactly the prime factors of the knot $K$.
\end{remark}

\section{Colored knotted graphs and orbifolds}
 Let $\cal C$ be a set of colors. (For example think of $\cal C=\N$.) By a {\em coloring} of a graph $G$ we mean a
 map $\varphi \colon E(G)\to \cal C$, where $E(G)$ is the set of all
 edges of $G$.

 \begin{definition} \label{coloredpair} Let $G_\varphi$ be a colored graph in a
 3--manifold $M$. Then the pair $(M,G_\varphi)$ is called {\em admissible}, if
 there is
 no incompressible sphere in $(M, G_\varphi)$ which crosses $G_\varphi$ transversely in two points of different
 colors.
 \end{definition}
 It follows from the definition that if $(M,G_\varphi)$ is admissible,
 then $G_\varphi$ has no
 valence two vertices    incident to edges of different colors.
We define   reductions along admissible spheres, trivial pairs,
  roots, efficient systems,   spherical slidings, and efficient roots just in the
 same way as for the uncolored case.

 \begin{theorem} \label{maincol} For any admissible pair  $(M,G_\varphi )$
   the
 root   exists and is unique up to color preserving homeomorphisms and removing
 trivial pairs. Moreover, any two efficient systems in $(M,G_\varphi )$ are
 equivalent and thus the efficient root is unique up to  color preserving homeomorphisms.
\end{theorem}

\begin{proof}  The proof is literally the same as for the
uncolored case. There is only one place where one should take into
account colorings:
  the last  paragraph   of the proof of \fullref{edgequivKG}.
  Indeed,  in this  paragraph    there appears
an incompressible  sphere   $S_e'$  that
 crosses $G$  in  two points.
We need to know that these points have the same colors, and
exactly for that purpose one has imposed the restriction that the
pair $(M,G_\varphi)$ must be  admissible.     \end{proof}

Further generalizations of the above result consist in specifying    sets
of {\em allowed}  single colors, pairs of colors, and triples of colors.
The idea is to define admissibility of spheres according to whether their
intersection with $G_\varphi$ belongs to one of the specified sets and
allow reductions only along those spheres.  Again, all proofs, in
particular, the proof of the corresponding version of
\fullref{maincol},
are literally the same with only
one exception
where we need that $(M,G_\varphi)$ is admissible. We
naturally obtain a generalized version of the {\em orbifold splitting theorem}
proved recently by C~Petronio~\cite{petr}.

Recall that a 3--orbifold can be described as a pair $(M,G_\varphi)$, where
all vertices of $G_\varphi$ have valence smaller than or equal to 3, all
univalent vertices are in $\partial M$,  and $G_\varphi$ is colored by the
set $\cal C$ of all integer numbers greater than 1.
  We specify the following sets of colors:

We allow no single colors, ie,  we do not perform reductions along
spheres crossing $G_\varphi$ at a single point. The set of allowed pairs
consists of pairs $(n,n), n\ge 2$. The allowed triples are the following:
$(2,2,n), n\geq 2$, and $(2,3,k), 3\leq k\leq 5$. See~\cite{petr} for
background. An orbifold $(M,G_\varphi )$ is called admissible, if it is
admissible in the above sense, ie, if there is no incompressible sphere
in $(M, G_\varphi)$ which crosses $G_\varphi$ transversely in two points
of different colors.

In view of the previous discussion, the following theorem is an easy
consequence of \fullref{maincol}.

  \begin{theorem} \label{orb} For any admissible orbifold  $(M,G_\varphi )$
   the
 root   exists and is unique up to orbifold homeomorphisms and removing
 trivial pairs. Moreover, any two efficient systems in $(M,G_\varphi )$ are
 equivalent and thus the efficient root is
  unique up to orbifold homeomorphism.\endproof
\end{theorem}

\section{Roots versus prime decompositions }
In \fullref{milnor}
we have seen that   the existence and
uniqueness of roots of 3--manifolds with respect to spherical
reductions (that is, {\em spherical roots}) is very close to the
existence and uniqueness of prime decompositions with respect to
connected sums. Indeed, to get the disjoint union of prime factors
one should merely perform as long as  possible spherical
reductions along {\em separating} essential spheres. The same  is
true   for prime decompositions of irreducible manifolds with
respect to boundary connected sums and for prime decompositions of
knotted graphs and orbifolds with respect to taking connected sums
(which are inverse operations to reductions along separating
essential spheres).

In contrast to that, the uniqueness of prime factors requires
additional arguments.  For  decompositions of 3--manifolds into
connected sums and boundary connected sums such arguments are
given in Sections \ref{S-roots}
and~\ref{D-roots}.
The uniqueness
of prime decompositions of knotted graphs  is not known yet.

\begin{Conjecture}Any knotted graph  is a connected sum of
prime factors. The  factors   are determined uniquely up to
homeomorphism.
\end{Conjecture}

The  uniqueness of prime decompositions of orbifolds  is also unsettled.
The main result of~\cite{petr} does not solve the problem, since it works
only for orbifolds without non-separating 2--suborbifolds. So the following
conjecture remains unsettled.

\begin{Conjecture}Any 3--dimensional orbifold  is a
connected sum of prime factors. The  factors   are determined
uniquely up to homeomorphism.
\end{Conjecture}

\section{Annular roots of manifolds} In addition to spherical and
and disc ($S$-- and $D$--) reductions from
Sections~\ref{milnor} and~\ref{D-roots}
we introduce $A$--reductions (reductions along   annuli).
\begin{definition}
 Let $A$ be an   annulus
  in  $M$ such that its boundary
circles lie in different components of $\partial M$. Then we cut
$M$ along $A$ and attach two plates $D_1^2{\times} I, D_2^2{\times} I$
by identifying their base annuli $\partial D_1^2{\times} I,
\partial D_2^2{\times} I$  with the two copies of $A$, which appear
under cutting.
\end{definition}

A reduction along an annulus $A\subset M$ is called {\em trivial}, if $A$
is  compressible, and non-trivial otherwise. Note that incompressible
annuli having boundary circles in different components of $\p M$ are
automatically essential, ie, not only incompressible, but also  boundary
incompressible.  It makes little sense to consider annular reductions
separately from spherical and disc ones. We will use them together calling
them $SDA$--{\em reductions}. As above, we begin by introducing an oriented graph
$\Gamma$. The set of vertices of $\Gamma$ is defined to be the set of all
compact 3--manifolds, but considered up to homeomorphism and removing
connected components homeomorphic to $S^3$ or $D^3$. Two vertices
$M_1,M_2$ are joined by an edge oriented from $M_1$ to $M_2$ if $M_2$ can
be obtained from $M_1$ by a non-trivial $SDA$--reduction.

Our next goal is to prove that $\Gamma$ possesses properties (CF)
and (EE). The inductive proof of property (EE) is based on the
following lemma, which
  helps us to settle the base of induction.

\begin{lemma}\label{anntrick}
 Let a 3--manifold $N$ be obtained from a 3--manifold $M$
 by annular reduction along  a compressible  annulus $A\subset M$.
   Then $N$ contains a sphere $S$ such that the
   spherical reduction $N_S$ of $N$ can be obtained from $M$ by
   cutting along two disjoint proper discs.
 \end{lemma}
 \begin{proof}
  Let  $D $ be a
 compressing disc for $A$.
 Denote by $U$ a closed regular neighborhood of $A\cup D$ in $M$. Then the
  relative boundary $\partial_{rel}U=\Cl(\partial N\cap \Int M)$
   consists of a parallel
 copy of $A$ and two proper discs $D', D''$. Denote by $S$
 a 2--sphere in
 $N$ composed from a copy of $D$ and a core disc of one of
 the attached
 plates, see \fullref{ddas}. 
 Then it is easy to see that the manifold
 $N_S$ obtained from $N$ by spherical reduction along $S$ is
 homeomorphic to  the result of cutting $M$ along $D'$ and $D''$.
\end{proof}

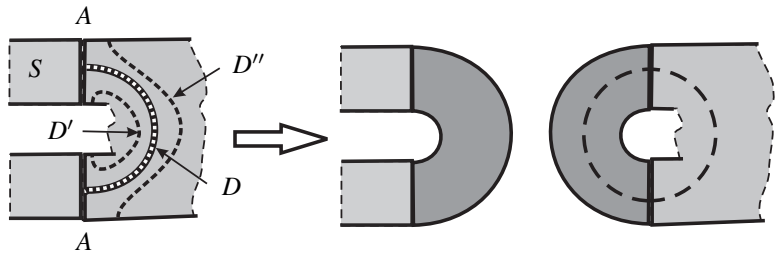
\begin{figure}[h]
\labellist
\small
\pinlabel {$S$} [l] at 100 692
\pinlabel {$A$} [l] at 127 595
\pinlabel {$A$} [l] at 127 722
\pinlabel {$D$} [l] at 209 626
\pinlabel {$D'$} [l] at 111 658
\pinlabel {$D''$} [l] at 215 695
\endlabellist
\centerline{\includegraphics[height=3.5cm]{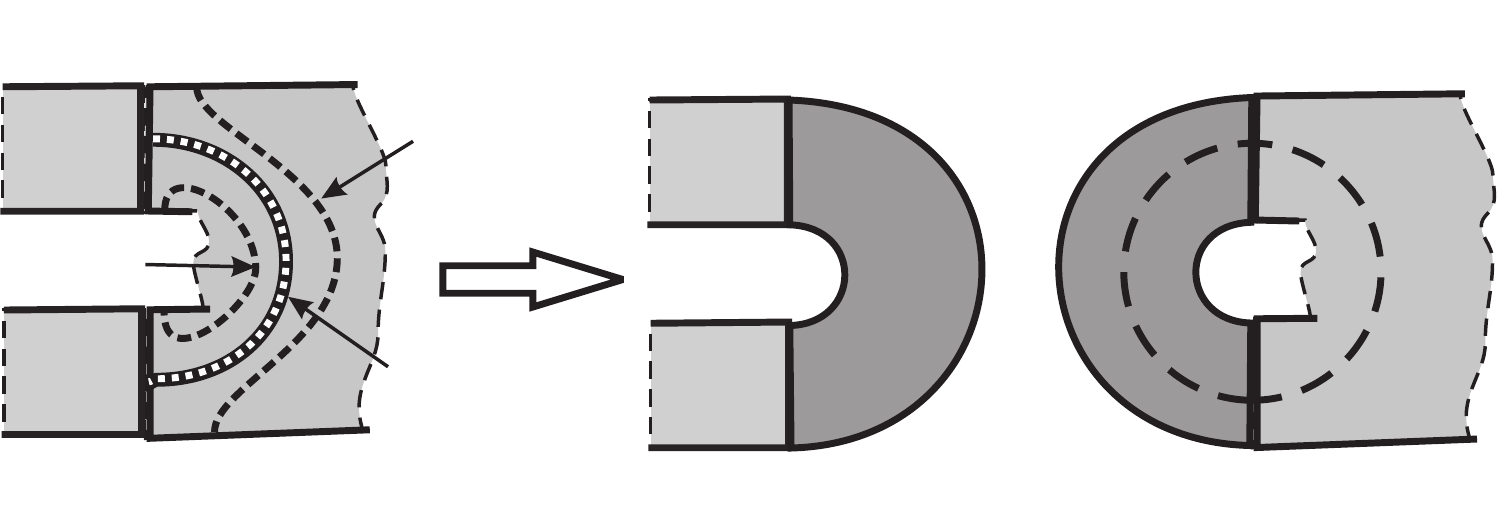}}
  \caption{Reduction along compressible annulus }
  \label{ddas}
\end{figure}

\begin{lemma}\label{lm:all equiv}
$\Gamma$ possesses properties (CF) and
 (EE).
\end{lemma}

\begin{proof} Property (CF) is easy:\\ The complexity function
$c(M)=(g^{(2)}(\partial M), s(M))$ introduced in
the proof of \fullref{edgequivD}
works.

To prove property (EE), consider two
 surfaces in $M$   corresponding to
edges $e,d$ of $\Gamma$. Each surface is either a non-trivial
sphere or disc, or an incompressible annulus having boundary
circles in different components of $\p M$. We  prove
 the equivalence $e\sim d$ by induction on the number
 $ \# (F_e\cap F_d)$ of curves
 in the intersection assuming that the  surfaces have been
 isotopically shifted so that this number   is minimal.

 {\bf Base of induction}\qua  Let $m=0$, ie~$ F_e\cap F_d=\emptyset$.
Denote by $M_e,M_{d}$ the manifolds obtained by reducing $M$
 along $F_e,F_d$, respectively. Since $ F_e\cap F_d=\emptyset$,
 $F_d$ survives the reduction along $F_e$ and thus may be
 considered as
 a surface in $F_e$. Let $N$ be obtained by  reduction of
 $F_e$ along $F_d$. Of course,
 reduction of $ M_{d} $  along  $F_e $
 also gives $ N $.   We claim that there is a  root $R$ of $N$ which
 is a
 common root of  $ M_{e} $  and $ M_{d}$ (and hence the edges
 $e,d$ are
 elementary equivalent).

  Indeed, if both surfaces
$F_e\subset M_d,F_d\subset M_e $ are non-trivial, then any root of $N$ is a
common root of $M_e$ and $M_d$. If one of them is a trivial sphere or a
trivial disc, then the same tricks as in the proofs of
Lemmas~\ref{edgequiv}
and~\ref{edgequivD}
do work.

Suppose that one of the surfaces (let $F_e$) is a trivial (ie
compressible) annulus. Then we apply \fullref{anntrick}
and get
a manifold $N'$ such that any root of $N'$ is a common root of
$M_e$ and $M_d$.

Now we suppose that both annuli $F_e,F_d$ are trivial. Then
we apply  \fullref{anntrick}
twice: construct two manifolds $N', N''$ such that each can be obtained
from $M_e$ (respectively, $M_d$) by cutting along two disjoint discs.
Since both $N',N''$ are spherical reductions of $N$, they have a common
root by \fullref{edgequiv}.

 {\bf Inductive assumption}\qua Any two edges $e,d$ with
$\# (F_e\cap F_d)\leq m$ are equivalent.

{\bf Inductive step}\qua Suppose that $ \# (F_e\cap F_d)\leq m+1$.
  We may assume that $F_e\cap F_d$ contains no trivial circles and
trivial arcs. Otherwise we could apply an innermost circle or an outermost
arc argument just as in the proof of Lemmas~\ref{edgequiv}
and~\ref{edgequivD}.
It follows that $F_e$ and $ F_d$ are annuli such
that $F_e\cap F_d$ consists either of non-trivial circles (which are
parallel to the core circles of the annuli) or of non-trivial arcs (which
join different boundary circles of the annuli).

 First we  suppose that  $F_e\cap F_d$ consists of non-trivial circles of $F_e$ and $F_d$.
 Then one can find two different components $A, B$ of $\partial M$
 such that
 a circle of $\partial F_e$ is in $A$ and a circle
 of $\partial F_d$ is in $B$. Denote by $s$
  the first circle of $F_e\cap F_d$ we meet at our radial
  way along $F_e$ from the
  circle $\partial F_e\cap A$ to the other boundary circle of $F_e$.
  Let $F_e'$ be the subannulus
  of $F_e$ bounded by   $\partial F_e\cap A$ and $s$, and
     $F_d'$   the subannulus of $F_d$ bounded by
         $s$ and $\partial F_d\cap B$.
      Then the
 annulus $F_e'\cup F_d'$ is essential and is isotopic to an
 annulus $X$ such that $\#(X\cap F_e)<m$ and
 $\#(X\cap F_d)=0$, see \fullref{newannulus}
  (to get a real picture, multiply by
 $S^1$). It follows from the inductive assumption
 that $ e\sim x \sim  d$,
 where $x$ is the edge corresponding to the annulus $X$.

\begin{figure}[ht]
\labellist
\small
\pinlabel {$F_d$} [l] at 224 661
\pinlabel {$F_e$} [l] at 286 527
\pinlabel {$X$} [l] at 330 683
\pinlabel {$A$} at 286 746
\pinlabel {$B$} [l] at 404 642
\endlabellist
\centerline{\includegraphics[height=5.6cm]{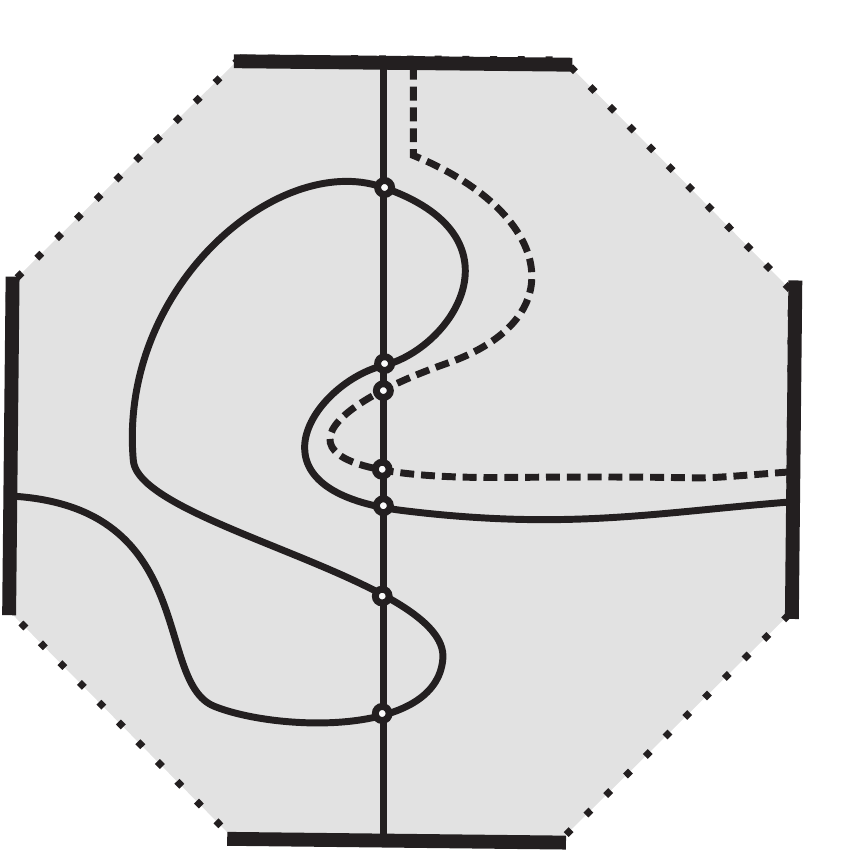}}
  \caption{$X$ is disjoint with  $F_d$ and
  crosses $F_e$ in a smaller number of circles.}
  \label{newannulus}
\end{figure}
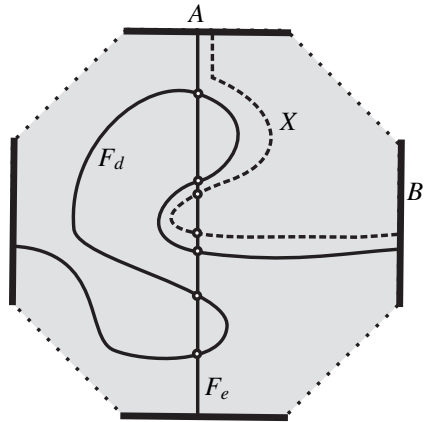

  Now we suppose  that $F_e\cap F_d$ consists
  of more than one  radial segments, each having endpoints in
  different  components of $\partial F_e$ and
  different  components of $\partial F_d$.
Let $s_1,s_2  \subset F_e\cap F_d\subset F_e$ be
    two neighboring segments.  Denote
    by $D$ the
quadrilateral part  of $F_e$
     between them.

  {\bf Case 1}\qua  First we assume that
    $F_d$ crosses $F_e$ at  $s_1,s_2$  in opposite directions.
    This means that each part of
    $F_d\setminus (s_1\cup s_2) $ approaches $D$
     from the same side. Then we cut
        $F_d$ along $s_1,s_2$ and attach to it
         two parallel copies of $D$
         lying on different
        sides of $F_e$. We get a new surface $F_d'$
          consisting of two disjoint annuli, at least
 one of which (denote it by $X$) is essential, see \fullref{lick}
 to the left. The real picture showing the behavior of the
 annuli in a neighborhood of $D$ can be obtained by multiplying
by $I$. Since $\#(X\cap F_e)\leq m-2$ and, after a small isotopy of $X$,
$\#(X\cap
 F_d)=\emptyset$, we get  $ e\sim x \sim  d$,
 where $x$ is the edge corresponding to the annulus $X$.

\begin{figure}[h]
\labellist
\small
\pinlabel {$s_1$} [l] at 130 619
\pinlabel {$s_1$} [l] at 345 620
\pinlabel {$s_2$} [l] at 102 705
\pinlabel {$s_2$} [l] at 315 705
\pinlabel {$F_d$} [l] at 52 705
\pinlabel {$F_d$} [l] at 263 704
\pinlabel {$F_e$} [l] at 104 591
\pinlabel {$F_e$} [l] at 319 591
\pinlabel {$X$} [l] at 402 707
\endlabellist
\centerline{\includegraphics[height=3.6cm]{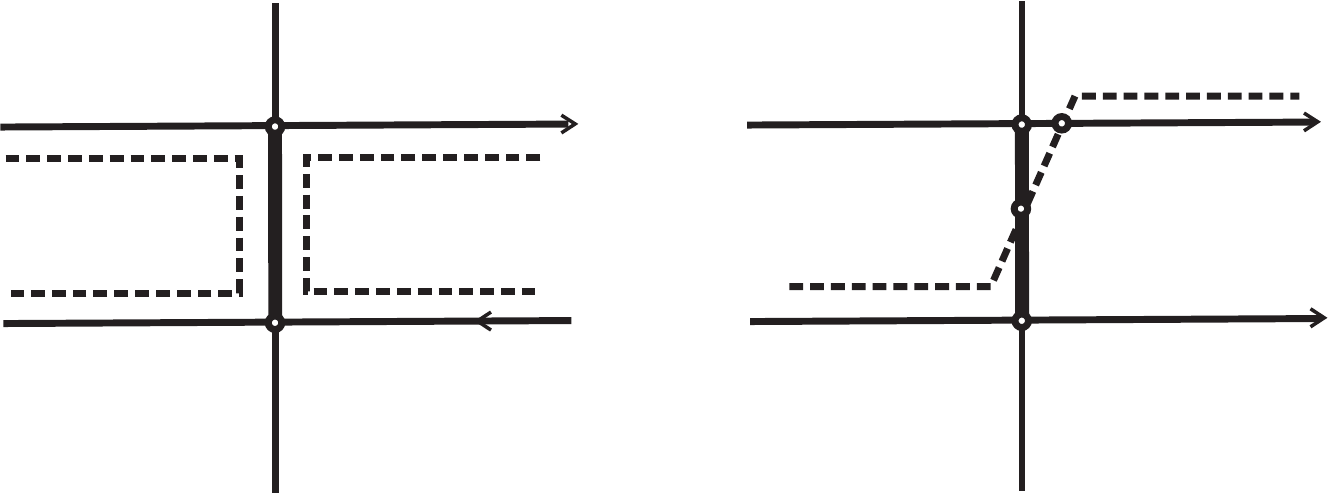}}
  \caption{$X$ crosses each annulus in a smaller number of segments.}
  \label{lick}
\end{figure}
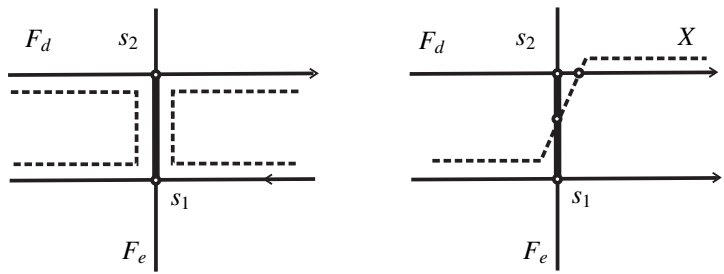

{\bf Case 2}\qua We assume now   that at all segments $F_d$ crosses $F_e$ at
$s_1,s_2$ in the same direction (say, from left to right). Then $s_1,s_2$
decompose $F_d$ into two strips $L_1,L_2$ such that $L_1$ approaches $s_1$
from the left side of $F_e$ and   $s_2$ from the  right side. Then the
annulus $L_1\cup D$ is isotopic to an annulus $X$ such that $\#(X\cap
F_e)\leq m-1$ and $\#(X\cap F_d)=1$, see \fullref{lick}
to the right.
Since $X$ crosses $F_e$ one  or more times in the same direction, it is
essential.
 Therefore, we get  $ e\sim x \sim  d$ again.

{\bf Case 3}\qua   Suppose   $M$ is not homeomorphic to $S^1{\times}
S^1{\times} I$ and $F_e$ and $F_d$ are annuli such that $F_e\cap
F_d$ consists of one radial segment. Denote by $F_d'$ the relative
boundary $\partial_{rel}(N)=\Cl(\p N\cap \Int M)$ of a regular
neighborhood $N=N(F_e\cup F_d)$ in $M$. Then $F_d'$ is an annulus
having boundary circles in different components of $\partial M$.

{\bf Case 3.1}\qua If $F_d'$ is incompressible, then we put $X=F_d'$.

{\bf Case 3.2}\qua   If $F_d'$ admits a compressing disc $D$,   then
the relative boundary of a regular neighborhood $N=N(F_d'\cup D)$
consists of a parallel copy of $F_d'$ and two proper discs
$D_1,D_2$. If at least one of these discs (say, $D_1$) is
essential, then we put $X=D_1$.

 {\bf Case 3.3}\qua Suppose that the  discs $D_1, D_2$ are not
 essential. Then the circles $\partial D_1, \partial D_2$ bound discs
 $D_1',D_2'$ contained in the
 corresponding components of $\partial M$.  We
 claim that at least one of the spheres $S_1=D_1\cup D_1',S_2=D_2\cup D_2'
 $ (denote it by $X$)  must be essential.
  Indeed, if both bound balls, then $M$ is
 homeomorphic to $S^1{\times} S^1{\times} I$, contrary to  our  assumption.

 In all three cases 3.1--3.3, $X$ is disjoint to $F_e$ as well as
 to $F_d$. Therefore,  $ e\sim x \sim  d$,
 where $x$ is the edge corresponding to the annulus $X$.

{\bf Case 4}\qua This is the last logical possibility. Suppose that
$M= S^1{\times} S^1{\times} I$. Then $e\sim d$ since reducing $M=
S^1{\times} S^1{\times} I$ along any incompressible  annulus having
boundary circles in different components of $M$ produces the same
manifold  $S^2{\times} I$.
\end{proof}

\begin{theorem} \label{annuli} For any 3--manifold $M$
the $SDA$--root $R(M)$ exists and is unique up to homeomorphism and
removal of spherical and ball connected components.
\end{theorem}

\begin{proof} This follows from \fullref{main} and
\fullref{lm:all equiv}.
\end{proof}

  It turns out that the condition on boundary circles of   annuli to lie in
different components of $\partial M$ is essential. Below we
present an example of a 3--manifold $M$ with two incompressible
boundary incompressible annuli $A,B\subset M$ such that $\partial
M$ is connected and  reductions of $M$ along $A$ and along $B$
lead us to two different 3--manifolds admitting no further
essential reduction, ie to two different  ``roots''.

\begin{Examplenn}Let $Q$ be the complement space of the figure-eight
knot. We assume that the torus $\partial Q$ is equipped with a coordinate
system such that the slope of the meridian is (1,0). Choose two pairs
$(p,q)$, $(m,n)$ of coprime integers such that $|q|,|n|\geq 2$ and
$|p|\neq |m|$. Let $a$ and $b$ be corresponding curves in $\partial Q$.
Then the manifolds $Q_{p,q}$ and $Q_{m,n}$ obtained by Dehn filling of $Q$
are not homeomorphic. By Thurston~\cite{Th}, they are hyperbolic.

Consider the  thick  torus $X=S^1{\times} S^1 {\times} I$ and locate
its  exterior meridian\break  $\mu=S^1{\times} \{\ast\}{\times} \{ 1\}$ and
 interior  longitude  $\lambda= \{\ast\}{\times} S^1{\times} \{ 0\}$.
Then we  attach to $X$   two copies $Q',Q''$ of $Q$ as follows.
The first copy $Q'$ is attached to $X$ by identifying an annular
regular neighborhood $N(a)$ of $a$ in $\partial Q$ with an annular
regular neighborhood $N(\mu)$ of  $\mu$ in $\partial X$. The
second copy $Q''$  is attached by identifying $N(b)$ with
$N(\lambda)$. Denote by $M$ the resulting manifold $Q'\cup X\cup
Q''$.

Since $Q$ is hyperbolic,  $M$ contains only two incompressible
boundary incompressible annuli $A$ and $B$, where $A$ is the
common image of $N(a)$ and $N(\mu)$, and  $B$ is the common image
of $N(b)$ and $N(\lambda)$.  It is easy to see that   reduction of
$M$ along $A$ gives us a disjoint union of a punctured $Q'_{p,q}$
and a punctured $Q''$ while the reduction along $B$ leaves us with
a punctured $Q'$ and a punctured $Q''_{m,n}$. After filling the
punctures (by reductions along spheres surrounding them), we get
two different  manifolds, homeomorphic to $Q _{p,q}\cup Q$ and
$Q_{m,n}\cup Q$. Since their connected components (ie $Q_{p,q},
Q_{m,n}, Q$) are hyperbolic, they are irreducible, boundary
irreducible and contain no essential annuli. Hence   $Q _{p,q}\cup
Q$ and $Q_{m,n}\cup Q$
 are different roots of $M$.
\end{Examplenn}

\section{Other roots}

 {\bf Roots of cobordisms}\qua  Recall that a {\em 3--cobordism} is a triple
 $(M,\partial_-M,\partial_+M)$, where $M$ is a compact 3--manifold and
 $\partial_-M$, $\partial_+M$ are unions of connected components of $\partial M$
 such that
 $\partial_-M\cap \partial_+M=
 \emptyset$ and $\partial_-M\cup \partial_+M=\partial M$. One can
 define $S$-- and $D$--reductions on cobordisms just in the same way as
 for manifolds. The $A$--reduction on cobordisms differs from
 the one for manifolds only in that one boundary circle of $A$
     must lie in $\partial_-M$ while the other   in $\partial_+M$.

\begin{theorem} \label{th:rootsco} For any compact 3--cobordism
$(M,\partial_-M,\partial_+M)$    its
 root   exists and is unique up to homeomorphism of cobordisms and
removing disjoint 3--spheres and balls.
\end{theorem}

The proof of this theorem is the same as the proof of \fullref{annuli}.
We point out that considering roots of cobordisms was motivated by the
paper~\cite{Ga} of S~Gadgil, which is interesting although the proof of
his main theorem contains a serious gap. We found the gap after proving
\fullref{th:rootsco},
which clarifies the situation with Gadgil's construction.

{\bf Roots of virtual links}\qua Recall that a virtual link can be defined as
a link   $L\subset F{\times} I$, where $F$ is a closed orientable surface.
Virtual links are considered up to isotopy and destabilization operations,
which,  in our terminology, correspond to   reduction along annuli. Each
annulus must be disjoint to $L$ and have one boundary circle in $F{\times}
\{ 0\}$, the other in $F{\times} \{ 1\}$. We also allow spherical
reductions. The proof of the following theorem is the same as the proof of
\fullref{annuli}.

\begin{theorem} \label{virtu} For any virtual link its root
 exists and is unique up to homeomorphism of cobordisms and
removing disjoint 3--spheres and balls.
\end{theorem}
This theorem is equivalent to the main theorem of Kuperberg~\cite{greg}.

\bibliographystyle{gtart}
\bibliography{link}

\begin{thebibliography}{}
\providecommand\bibmarginpar{\leavevmode\marginpar}
\def\urlstyle#1{{\tt #1}}

\bibitem{Bo}
\textbf{F Bonahon},
  \href{http://www.numdam.org/item?id=ASENS_1983_4_16_2_237_0} {\emph{Cobordism
  of automorphisms of surfaces}}, Ann. Sci. \'Ecole Norm. Sup. $(4)$ 16 (1983)
  237--270 \xox{MR}{732345}

\bibitem{Ga}
\textbf{S Gadgil}, \emph{On the Andrews-Curtis conjecture and algorithms from
  topology} \xox{arXiv}{math.GT/0108116}

\bibitem{Kn}
\textbf{H Kneser}, \emph{Geschlossene Fl\"achen in dreidimensionalen
  Mannigfaltigkeiten}, Jahresbericht der Deutsche Mathematische Vereinigung 38
  (1929) 248--260 \xox{MR}{1997331}

\bibitem{greg}
\textbf{G Kuperberg}, \href{http://dx.doi.org/10.2140/agt.2003.3.587}
  {\emph{What is a virtual link?}}, Algebr. Geom. Topol. 3 (2003) 587--591
  \xox{MR}{1997331}

\bibitem{mpim2}
\textbf{S Matveev}, \emph{Roots of knotted graphs and orbifolds}, MPIM-Preprint
  51 (2005) \xox{arXiv}{math.GT/0504415v1}

\bibitem{Mil}
\textbf{J Milnor}, \href{http://dx.doi.org/10.2307/2372800} {\emph{A unique
  decomposition theorem for {$3$}-manifolds}}, Amer. J. Math. 84 (1962) 1--7
  \xox{MR}{0142125}

\bibitem{petr}
\textbf{C Petronio}, \href{http://dx.doi.org/10.1017/S0305004106009807}
  {\emph{Spherical splitting of 3-orbifolds}}, Math. Proc. Cambridge Philos.
  Soc. 142 (2007) 269--287 \xox{MR}{2314601}

\bibitem{swar}
\textbf{G\,A Swarup}, \emph{Some properties of {$3$}-manifolds with boundary},
  Quart. J. Math. Oxford Ser. $(2)$ 21 (1970) 1--23 \xox{MR}{0276986}

\bibitem{Th}
\textbf{W\,P Thurston}, \emph{The geometry and topology of 3-manifolds},
  mimeographed notes, Princeton University (1979)

\end{thebibliography}

\end{document}